\documentclass[10pt,a4paper]{amsart}
\usepackage[utf8]{inputenc}
\usepackage[english]{babel}
\usepackage{amsmath}
\usepackage{amsfonts}
\usepackage{amssymb}
\usepackage{amsthm}
\usepackage{graphicx}
\usepackage{epsfig}
\usepackage[pagewise]{lineno} 

\usepackage{dsfont}
\usepackage[all]{xy}

\newtheorem{thm}{Theorem}

\newtheorem{defi}{Definition}
\newtheorem{prop}{Proposition}
\newtheorem{rem}{Remark}

\newcommand{\Sp}{\mathbb{S}}
\newcommand{\Hy}{\mathbb{H}}

\newcommand{\R}{\mathbb{R}}
\newcommand{\pl}{\langle}
\newcommand{\pr}{\rangle}

\newcommand{\Qu}{\mathrm{H}}

\newcommand{\Ph}{\phi_{\alpha,\beta}(t)}

\title{Helicoidal flat surfaces in the 3-sphere} 

\author[F. Manfio]{F. Manfio$^1$} \thanks{$^1$Research partially supported by FAPESP/Brazil, grant 2014/01989-9}
\address{ICMC, Universidade de S\~ao Paulo, S\~ao Carlos, Brasil} 
\email{manfio@icmc.usp.br}

\author[J. P. dos Santos]{J. P. dos Santos$^2$} \thanks{$^2$Research partially supported by FEMAT/UnB and FAPDF}      
\address{MAT, Universidade de Bras\'ilia, Bras\'ilia, Brazil}
\email{joaopsantos@unb.br}


\begin{document}

\begin{abstract}
In this paper, helicoidal flat surfaces in the $3$-di\-men\-sio\-nal
sphere $\Sp^3$ are considered. A complete classification of such
surfaces is given in 
terms of their first and second fundamental forms and by linear solutions
of the corresponding angle function. The classification is obtained by 
using the Bianchi-Spivak construction for flat surfaces and a 
representation for constant angle surfaces in $\Sp^3$.
\end{abstract}

\maketitle

\section{Introduction}

Helicoidal surfaces in $3$-dimensional space forms arise as a natural
generalization of rotational surfaces in such spaces. These surfaces
are invariant by a subgroup of the group of isometries of the ambient
space, called helicoidal group, whose elements can be seen
as a composition of a translation with a rotation for a given axis.

In the Euclidean space $\R^3$, do Carmo and Dajczer \cite{docarmo} describe the 
space of all helicoidal surfaces that have constant mean
curvature or constant Gaussian curvature. This space behaves as
a circular cylinder, where a given generator corresponds to the 
rotational surfaces and each parallel corresponds to a periodic family of helicoidal surfaces. Helicoidal surfaces with prescribed mean or Gaussian curvature are obtained by Baikoussis and Koufogiorgos \cite{BK}. More precisely, they obtain a closed form of such a surface by integrating the second-order ordinary differential equation satisfied by the generating curve of the surface. Helicoidal 
surfaces in $\R^3$ are also considered by Perdomo \cite{P1} in the context
of minimal surfaces, and by Palmer and Perdomo \cite{PP2} where the
mean curvature is related with the distance to the $z$-axis. In the 
context of constant mean curvature, helicoidal surfaces are considered by
Solomon and Edelen in \cite{edelen1}.

In the $3$-dimensional hyperbolic space $\Hy^3$, Mart\'inez, the second 
author and Tenenblat \cite{MST} give a complete classification of the 
helicoidal flat surfaces in terms of meromorphic data, which extends the 
results obtained by Kokubu, Umehara and Yamada \cite{KUY} for rotational
flat surfaces. Moreover, the classification is also given by means of linear 
harmonic functions, characterizing the flat fronts in $\mathbb{H}^3$ that
correspond to linear harmonic functions. Namely, it is well known that for 
flat surfaces in $\mathbb{H}^3$, on a neighbourhood of a non-umbilical 
point, there is a curvature line parametrization such that the first and 
second fundamental forms are given by 
\begin{equation}
\begin{array}{rcl}
I &=& \cosh^2 \phi(u,v) (du)^2 + \sinh^2 \phi(u,v) (dv)^2,  \\
II &=& \sinh \phi(u,v) \cosh \phi(u,v) \left((du)^2 + (dv)^2 \right), 
\end{array} \label{firstff}
\end{equation}
where $\phi$ is a harmonic function, i.e., $\phi_{uu}+\phi_{vv}=0$. In 
this context, the main result states that a surface in $\Hy^3$, parametrized
by curvature lines, with fundamental forms as in \eqref{firstff} and 
$\phi(u,v)$ linear, i.e, $\phi(u,v) = a u + b v + c$, is flat if and only if, the
surface is a helicoidal surface or a {\em peach front}, where the second
one is associated to the case $(a,b,c) = (0,\pm1,0)$.
Helicoidal minimal surfaces were studied by Ripoll \cite{ripoll} and
helicoidal constant mean curvature surfaces in $\mathbb{H}^3$ are 
considered by Edelen \cite{edelen2}, as well as the cases where 
such invariant surfaces belong to $\mathbb{R}^3$ and $\mathbb{S}^3$.

Similarly to the hyperbolic space, for a given flat surface in the
$3$-dimensional sphere $\Sp^3$, there exists a parametrization by
asymptotic lines, where the first and the second fundamental forms
are given by
\begin{equation} 
\begin{array}{rcl}
I &=& du^2 + 2 \cos \omega du dv + dv^2, \\
II &=& 2  \sin \omega du dv \label{primeiraffs}
\end{array} 
\end{equation}
for a smooth function $\omega$, called the {\em angle function}, that 
satisfy the homogeneous wave equation $\omega_{uv} = 0$. 
Therefore, one can ask which surfaces are related to linear solutions
of such equation. 

The aim of this paper is to give a complete classification of helicoidal flat
surfaces in $\Sp^3$, established in Theorems 1 and 2, by means of asymptotic lines coordinates, with first
and second fundamental forms given by \eqref{primeiraffs}, where the 
angle function is linear. In order to do this, one uses the Bianchi-Spivak construction for flat surfaces in $\Sp^3$. This construction and the Kitagawa representation \cite{kitagawa1}, are important tools used
in the recent developments of flat surface theory. Examples of applications
of such representations can be seen in \cite{galvezmira1} and \cite{aledogalvezmira}. Our classification also makes use of a 
representation for constant angle surfaces in $\Sp^3$, who comes 
from a characterization of constant angle surfaces in the Berger 
spheres obtained by Montaldo and Onnis \cite{MO}.

This paper is organized as follows. In Section 2 we give a brief description
of helicoidal surfaces in $\Sp^3$, as well as a ordinary differential equation
that characterizes those one that has zero intrinsic curvature. 

In Section 3, the Bianchi-Spivak construction is introduced. It will be used to prove Theorem 1, which states that a flat surface in $\Sp^3$, with asymptotic
parameters and linear angle function, is invariant under helicoidal motions. 

In Section 4, Theorem 2 establishes  the converse of Theorem 1, that is, a helicoidal flat surface admits a local parametrization,  given
by asymptotic parameters where the angle function is linear. Such local parametrization is obtained by using a characterization 
of constant angle surfaces in Berger spheres, which is a consequence of the fact that a helicoidal flat surface is a 
constant angle surface in $\Sp^3$, i.e., it has a unit normal that makes a constant angle with the Hopf vector field.

In section 5 we present an application for conformally flat hypersurfaces
in $\R^4$. The classification result obtained is used to give a geometric characterization for special conformally flat surfaces in $4-$dimensional
space forms. It is known that conformally flat hypersurfaces in 
$4$-dimensional space forms are associated with solutions of a system
of equations, known as Lam\'e's system (see \cite{Jeromin1} and 
\cite{santos} for details). In \cite{santos}, Tenenblat and the second author
obtained invariant solutions under symmetry groups of Lam\'e's system.
A class of those solutions is related to flat surfaces in $\Sp^3$, 
parametrized by asymptotic lines with linear angle function. Thus
a geometric description of the correspondent conformally flat 
hypersurfaces is given in terms of helicoidal flat surfaces in $\Sp^3$.

\section{Helicoidal flat surfaces}

Given any $\beta\in\R$, let $\{\varphi_\beta(t)\}$ be the one-parameter
subgroup of isometries of $\Sp^3$ given by
\[
\varphi_\beta(t) = \left( 
\begin{array}{cccc}
1  & 0 &  0 & 0 \\
0 & 1 & 0 & 0 \\
0 & 0 & \cos \beta t & -\sin \beta t \\
0 & 0 & \sin \beta t & \cos \beta t \\
\end{array} 
\right).
\]
When $\beta\neq0$, this group fixes the set $l=\{(z,0)\in\Sp^3\}$, which
is a great circle and it is called the {\em axis of rotation}. In this case, the
orbits are circles centered on $l$, i.e., $\{\varphi_\beta(t)\}$ consists of 
rotations around $l$. Given another number $\alpha\in\R$, consider now
the translations
$\{\psi_\alpha(t)\}$ along $l$,
\[
\psi_\alpha(t)=\left( 
\begin{array}{cccc}
\cos \alpha t & -\sin \alpha t &  0 & 0 \\
\sin \alpha t & \cos \alpha t & 0 & 0 \\
0 & 0 & 1 & 0 \\
0 & 0 & 0 & 1 \\
\end{array} 
\right).
\]

\begin{defi}\label{def:helicoidal}
{\em A {\em helicoidal} surface in $\Sp^3$ is a surface invariant under the 
action of the helicoidal 1-parameter group of isometries
\begin{equation}\label{eq:movhel}
\Ph=\psi_\alpha(t)\circ\varphi_\beta(t)=
\left( 
\begin{array}{cccc}
\cos \alpha t & -\sin \alpha t &  0 & 0 \\
\sin \alpha t & \cos \alpha t & 0 & 0 \\
0 & 0 & \cos \beta t & -\sin \beta t \\
0 & 0 & \sin \beta t & \cos \beta t \\
\end{array} 
\right),
\end{equation}}
{\em given by a composition of a translation $\psi_\alpha(t)$ and a rotation $\varphi_\beta(t)$ in $\Sp^3$. }
\end{defi}

\begin{rem}
{\em When $\alpha=\beta$, these isometries are usually called 
{\em Clifford translations}. In this case, the orbits are all great circles, 
and they are equidistant from each other. In fact, the orbits of the action
of $G$ coincide with the fibers of the Hopf fibration $h:\Sp^3\to\Sp^2$.
We note that, when $\alpha=-\beta$, these isometries are also, up to
a rotation in $\Sp^3$, Clifford translations. For this reason we will consider
in this paper only the cases $\alpha\neq\pm\beta$.}
\end{rem}

With these basic properties in mind, a helicoidal surface can be locally
parametrized by
\begin{equation}\label{eq:param-helicoidal}
X(t,s) = \Ph\cdot\gamma(s),
\end{equation}
where $\gamma:I\subset\R\to\Sp^2_+$ is a curve parametrized by the arc
length, called the {\em profile curve} of the parametrization $X$. Here, 
$\Sp^2_+$ is the half totally geodesic sphere of $\Sp^3$ given by
\[
\Sp^2_+ = \left\{(x_1, x_2, x_3, 0)\in\Sp^3: x_3>0\right\}.
\]
Then we have
\[
\begin{array}{rcl}
X_t &=& \Ph\cdot(-\alpha x_2,\alpha x_1, 0,\beta x_3), \\
X_s &=& \Ph\cdot\gamma'(s).
\end{array}
\]
Moreover, a unit normal vector field associated to the parametrization
$X$ is given by $N=\tilde N/ \|\tilde N\|$, where $\tilde N$ is explicitly given
by
\begin{equation}\label{normal-field}
\tilde{N} = \Ph\cdot \big(\beta x_3(x_2'x_3-x_2 x_3',\beta x_3(x_1x_3'-x_1'x_3),
\beta x_3 (x_1'x_2-x_1 x_2'),-\alpha x_3' \big).
\end{equation}
Let us now consider a parametrization by the arc length of  $\gamma$ 
given by
\begin{equation} \label{eq:param-gamma}
\gamma(s) = \big(\cos\varphi(s)\cos\theta(s),\cos\varphi(s)\sin\theta(s),
\sin\varphi(s),0\big). 
\end{equation}

We will finish this section discussing the flatness of helicoidal surfaces
in $\Sp^3$. Recall that a simple way to obtain flat surfaces in $\Sp^3$
is by means of the Hopf fibration $h:\Sp^3\to\Sp^2$. More precisely,
if $c$ is a regular curve in $\Sp^2$, then $h^{-1}(c)$ is a flat surface
in $\Sp^3$ (cf. \cite{spivak}). Such surfaces are called {\em Hopf cylinders}. The next result provides a necessary
and sufficient condition for a helicoidal surface, parametrized as in 
\eqref{eq:param-helicoidal}, to be flat.

\begin{prop}\label{prop:HSF}
A helicoidal surface locally parametrized as in \eqref{eq:param-helicoidal},
where $\gamma$ is given by \eqref{eq:param-gamma},
is a flat surface if and only if the following equation
\begin{equation}
\beta^2\varphi''\sin^3\varphi\cos\varphi - \beta^2(\varphi')^2 \sin^4\varphi 
+\alpha^2(\varphi')^4 \cos^4 \varphi = 0 \label{ode-helicoidal}
\end{equation} \label{prop-ode-helicoidal}
is satisfied. 
\end{prop}
\begin{proof}
Since $\Ph\in O(4)$ and $\gamma$ is parametrized by the arc length, the 
coefficients of the first fundamental form are 
given by
\[
\begin{array}{rcl}
E &=& \alpha^2\cos^2\varphi + \beta^2\sin^2\varphi, \\
F &=& \alpha\theta'\cos^2\varphi,\\
G &=& (\varphi')^2 + (\theta')^2 \cos^2 \varphi =1.
\end{array}
\]
Moreover, the Gauss curvature $K$ is given by
\[
\begin{array}{rcl}
4(EG - F^2)^2 K &=& E \left[ E_s G_s - 2 F_t G_s + (G_t)^2 \right] +
G \left[E_t G_t - 2 E_t F_s + (E_s)^2 \right]  \\
&&+ F (E_t G_s - E_s G_t - 2 E_s F_s + 4 F_t F_s - 2 F_t G_t )  \\
&& - 2 (EG-F^2)(E_{ss} - 2 F_{st} + G_{tt} ).
\end{array}
\]
Thus, it follows from the expression of $K$ and from the coefficients of
the first fundamental form that the surface is flat if, and only if, 
\begin{equation}\label{eq:gauss}
E_s (EG - F^2)_s - 2 (EG-F^2)E_{ss}=0. 
\end{equation}
When $\alpha=\pm\beta$, the equation \eqref{eq:gauss} is trivially
satisfied, regardless of the chosen curve $\gamma$. For the case
$\alpha\neq\pm\beta$, since 
\[
EG-F^2 =\beta^2\sin^2\varphi+\alpha^2(\varphi')^2\cos^2\varphi,
\] 
a straightforward computation shows that the equation \eqref{eq:gauss}
is equivalent to 
\[
(\beta^2-\alpha^2)\big(\beta^2\varphi''\sin^3\varphi\cos\varphi
-\beta^2(\varphi')^2\sin^4\varphi+\alpha^2(\varphi')^4\cos^4\varphi\big) = 0,
\]
and this concludes the proof.
\end{proof}

\section{The Bianchi-Spivak construction}

A nice way to understand the fundamental equations of a flat surface $M$ in
$\Sp^3$ is by parameters whose coordinate curves are asymptotic curves
on the surface. As $M$ is flat, its intrinsic curvature vanishes identically. 
Thus, by the Gauss equation, the extrinsic curvature of $M$ is constant 
and equal to $-1$. In this case, as the extrinsic curvature is negative, it is
well known that there exist Tschebycheff coordinates around every point.
This means that we can choose local coordinates $(u,v)$ such that the
coordinates curves are asymptotic curves of $M$ and these curves are
parametrized by the arc length. In this case, the first and second fundamental
forms are given by 
\begin{eqnarray}\label{eq:forms}
\begin{aligned}
I  &=  du^2 + 2 \cos \omega dudv + dv^2, \\
II&  =   2 \sin \omega du dv,
\end{aligned}
\end{eqnarray}
for a certain smooth function $\omega$, usually called the {\em angle function}.
This function $\omega$ has two basic properties. The first one is that as $I$
is regular, we must have $0<\omega<\pi$. Secondly, it follows from the Gauss
equation that $\omega_{uv}=0$. In other words, $\omega$ satisfies the homogeneous wave equation, and thus it can be locally decomposed as 
$\omega(u,v) = \omega_1(u) + \omega_2(v)$, where $\omega_1$ and 
$\omega_2$ are smooth real functions (cf. \cite{galvez1} and \cite{spivak}
for further details).

\vspace{.2cm}

Given a flat isometric immersion $f:M\to\Sp^3$ and a local smooth unit
normal vector field $N$ along $f$, let us consider coordinates $(u,v)$
such that the first and the second fundamental forms of $M$ are given
as in \eqref{eq:forms}. The aim of this work is to characterize the flat 
surfaces when the angle function  $\omega$ is linear, i.e., when
$\omega=\omega_1+\omega_2$ is given by
\begin{eqnarray}\label{eq:linear}
\omega_1(u) + \omega_2 (v) = \lambda_1 u + \lambda_2 v + \lambda_3
\end{eqnarray}
where $\lambda_1, \lambda_2, \lambda_3 \in \R$. I order to do this, let us
first construct flat surfaces in $\Sp^3$ whose first and second fundamental
forms are given by \eqref{eq:forms} and with linear angle function. This 
construction is due to Bianchi \cite{bianchi} and Spivak \cite{spivak}.

\vspace{.2cm}

We will use here the division algebra of the quaternions, a very useful approach
to describe explicitly flat surfaces in $\Sp^3$. More precisely, we identify
the sphere $\Sp^3$ with the set of the unit quaternions 
$\{q\in\Qu: q\overline q=1\}$ and $\Sp^2$ with the unit sphere in the
subspace of $\Qu$ spanned by $1$, $i$ and $j$.

\begin{prop}[Bianchi-Spivak representation]\label{teo:BS}
Let $c_a,c_b:I\subset\R\to\Sp^3$ be two curves parametrized by the 
arc length, with curvatures $\kappa_a$ and $\kappa_b$, and whose torsions
are given by $\tau_a=1$ and $\tau_b=-1$. Suppose that $0 \in I$, $c_a(0)=c_b(0)=(1,0,0,0)$ e $c_a'(0) \wedge c_b '(0) \neq 0$. 
Then the map
\[
X(u,v) = c_a (u) \cdot c_b (v)
\]
is a local parametrization of a flat surface in $\Sp^3$, whose first and 
second fundamental forms are given as in \eqref{eq:forms}, where the
angle function satisfies $\omega_1'(u) = -\kappa_a (u)$ and 
$\omega_2'(v) =\kappa_b (v)$.
\end{prop}

Since the goal here is to find a parametrization such that $\omega$ can
be written as in \eqref{eq:linear}, it follows from Theorem \ref{teo:BS}
that the curves of the representation must have constant curvatures. 
Therefore, we will use the Frenet-Serret formulas in order to obtain curves
with torsion $\pm 1$ and with constant curvatures.

\vspace{.2cm}

Given a real number $r>1$, let us consider the curve 
$\gamma_r:\R\to\Sp^3$ given by
\begin{equation}\label{eq:base}
\gamma_r(u) = \frac{1}{\sqrt{1+r^2}} \left(r\cos\frac{u}{r},r\sin\frac{u}{r}, 
\cos ru,\sin ru\right).
\end{equation}
A straightforward computation shows that $\gamma_r(u)$ is 
parametrized by the arc length, has constant curvature 
$\kappa=\frac{r^2-1}{r}$ and its torsion $\tau$ satisfies $\tau^2=1$.
Observe that $\gamma_r(u)$ is periodic if and only if $r^2\in\mathbb Q$.
When $r$ is a positive integer, $\gamma_r(u)$ is a closed
curve of period $2\pi r$. A curve $\gamma$ as in \eqref{eq:base} will 
be called a {\em base curve}.

\vspace{.2cm}

Now we just have to apply rigid motions to a base curve in order to satisfy
the remaining requirements of the Bianchi-Spivak construction. It is easy
to verify that the curves
\begin{eqnarray}\label{eq:curves-condition}
\begin{array}{rcl}
c_a(u) &=& \dfrac{1}{\sqrt{1+a^2}}(a,0,-1,0) \cdot \gamma_a (u), \\
c_b(v) &=& \dfrac{1}{\sqrt{1+b^2}}T ( \gamma_b (v)) \cdot (b,0,0,-1),
\end{array}
\end{eqnarray}
are base curves, and satisfy $c_a(0)=c_b(0)=(1,0,0,0)$ and 
$c_a'(0)\wedge c_b'(0)\neq0$, where
\begin{equation}
\begin{array}{rcl}
T &=& \left( \begin{array}{cccc}
1&0&0 & 0\\
0&1&0& 0 \\
0&0&0&1\\
0&0&1&0
\end{array} \right). \\
\end{array} 
\end{equation}
 
Therefore we can establish our first main result:

\begin{thm}\label{teo:main1}
The map $X:U\subset\R^2\to\Sp^3$ given by 
\[
X(u,v) = c_a(u) \cdot c_b(v),
\]
where $c_a$ and $c_b$ are the curves given in \eqref{eq:curves-condition},
is a parametrization of a flat surface in $\Sp^3$, whose first and second
fundamental forms are given by
\[
\begin{array}{rcl}
I &=& du^2 +2  \cos \left( \left( \frac{1-a^2}{a} \right) u + \left( \frac{b^2-1}{b} \right) v + c \right) du dv + dv^2, \\
II &=& 2 \sin \left(\left( \frac{1-a^2}{a} \right) u + \left( \frac{b^2-1}{b} \right) v + c \right) du dv,
\end{array}
\]
where $c$ is a constant. Moreover, up to rigid motions, $X$ is invariant under helicoidal motions. 
\end{thm}
\begin{proof}
The statement about the fundamental forms follows directly from the 
Bianchi-Spivak construction. For the second statement, note that
the parametrization $X(u,v)$ can be written as
\[
X(u,v) = g_a \cdot  Y(u,v) \cdot g_b,
\]
where
\[
\begin{array}{rcl}
g_a &=& \dfrac{1}{\sqrt{1+a^2}}(a,0,-1,0), \\
g_b &=& \dfrac{1}{\sqrt{1+b^2}}(b,0,0,-1), 
\end{array}
\]
and
\[
Y(u,v) = \gamma_a(u) \cdot T(\gamma_b(v)).
\]
To conclude the proof, it suffices to show that $Y(u,v)$ is invariant by
helicoidal motions. To do this, we have to find $\alpha$ and $\beta$
such that
\[
\Ph\cdot Y(u,v)=Y\big(u(t),v(t)\big),
\]
where $u(t)$ and $v(t)$ are smooth functions. Observe that $Y(u,v)$
can be written as
\begin{eqnarray}\label{eq:expY}
Y(u,v) = \dfrac{1}{\sqrt{(1+a^2)(1+b^2)}} (y_1, y_2, y_3, y_4),
\end{eqnarray}
where
\begin{equation*}
\begin{array}{rcl}
y_1(u,v) &=& ab \cos \left(\dfrac{u}{a} + \dfrac{v}{b} \right) - \sin (au+ bv), \\
y_2(u,v) &=& ab \sin \left(\dfrac{u}{a} + \dfrac{v}{b} \right) + \cos (au+ bv), \\
y_3(u,v) &=& b \cos \left(au- \dfrac{v}{b} \right) - a\sin\left(\dfrac{u}{a}-bv\right), \\
y_4(u,v) &=& b\sin\left(au-\dfrac{v}{b} \right)+a\cos\left(\dfrac{u}{a}-bv\right). \\
\end{array} \label{parametrizacao-y}
\end{equation*}
A straightforward computation shows that if $\Ph$ is given by 
\eqref{eq:movhel}, we have
\[
u(t) =u+z(t) \quad\textrm{and}\quad v(t)=v+w(t),
\]
where
\begin{eqnarray}\label{eq:z(t)w(t)}
z(t)=\frac{a(b^2-1)}{a^2b^2-1}\beta t \quad\text{and}\quad
w(t)=\frac{b(1-a^2)}{a^2b^2-1}\beta t,
\end{eqnarray}
with 
\begin{eqnarray}\label{eq:alpha-beta}
\alpha=\dfrac{b^2-a^2}{a^2b^2-1}\beta,
\end{eqnarray}
showing that $Y(u,v)$ is invariant by helicoidal motions. Observe that
when $a=\pm b$ we have $\alpha=0$, i.e., $X$ is a rotational surface
in $\Sp^3$. 
\end{proof}

\begin{rem}
{\em It is important to note that the constant $a$ and $b$ in 
\eqref{eq:curves-condition} were considered in $(1,+\infty)$ in order to 
obtain non-zero constant curvatures with its well defined torsions,
and then to apply the Bianchi-Spivak construction. This is not a strong
restriction since the curvature function $\kappa(t)=\frac{t^2-1}{t}$ assumes
all values in $\R \setminus\{0\}$ when $t \in (1, + \infty)$. However, by 
taking $a=1$ and $b>1$ in \eqref{eq:curves-condition}, a long but 
straightforward computation gives an unit normal vector field
\[
N(u,v) = \dfrac{1}{\sqrt{2(1+b^2)}} (n_1, n_2, n_3, n_4),
\]
where
\[
\begin{array}{rcl}
n_1(u,v) &=& -b \sin \left(u + \dfrac{v}{b} \right) + \cos (u+ bv), \\
n_2(u,v) &=& b \cos \left(u + \dfrac{v}{b} \right) + \sin (u+ bv), \\
n_3(u,v) &=& b \sin \left(u- \dfrac{v}{b} \right) - \cos\left(a-bv\right), \\
n_4(u,v) &=& -b\cos \left(u-\dfrac{v}{b} \right) - \sin\left(u-bv\right). \\
\end{array}
\]
Therefore, one shows that this parametrization is also by asymptotic 
lines where the angle function is given by 
$\omega(u,v) =\frac{1-b^2}{b}v -\frac{\pi}{2}$. Moreover, this is a
parametrization of a Hopf cylinder, since the unit normal vector field $N$
makes a constant angle with the Hopf vector field (see section 4).}
\end{rem}

We will use the parametrization $Y(u,v)$ given in \eqref{eq:expY}, compose with the stereographic projection in $\R^3$,  to visualize some examples with the corresponding constants $a$ and $b$.

\begin{figure}[!htb]
\begin{minipage}[]{0.45\linewidth}
\includegraphics[width=\linewidth]{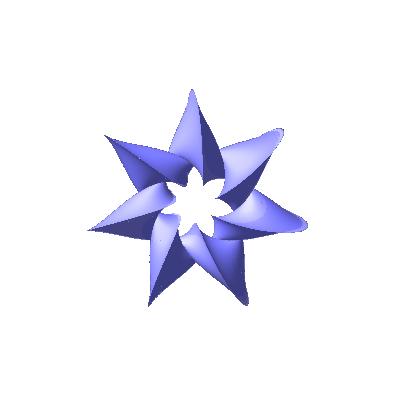}
\end{minipage}
\begin{minipage}[]{0.45\linewidth}
\includegraphics[width=\linewidth]{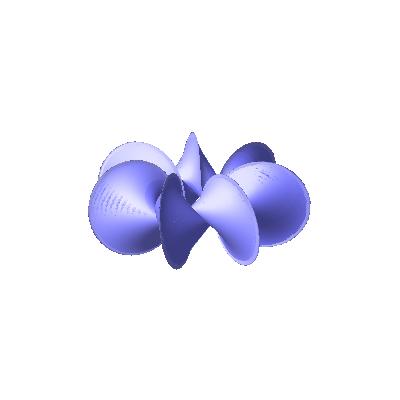}
\end{minipage}
\caption{$a=2$ and $b=3$.}
\end{figure}

\begin{figure}[!htb]
\begin{minipage}[]{0.45\linewidth}
\includegraphics[width=\linewidth]{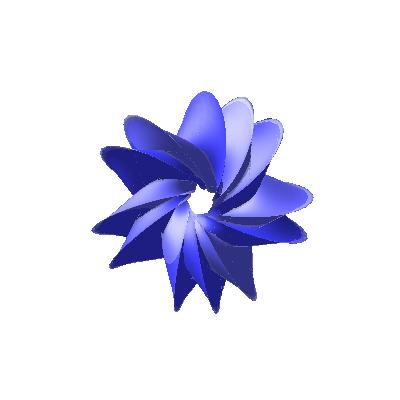}
\end{minipage}
\begin{minipage}[]{0.45\linewidth}
\includegraphics[width=\linewidth]{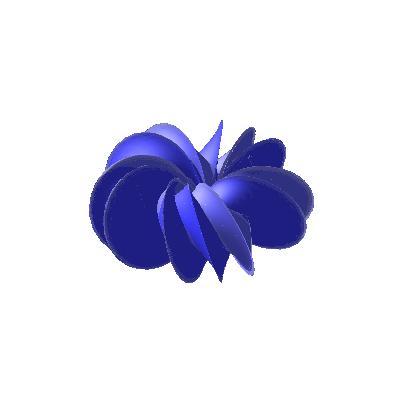}
\end{minipage}
\caption{$a=\sqrt2$ and $b=3$.}
\end{figure}

\begin{figure}[!htb]
\begin{minipage}[]{0.45\linewidth}
\includegraphics[width=\linewidth]{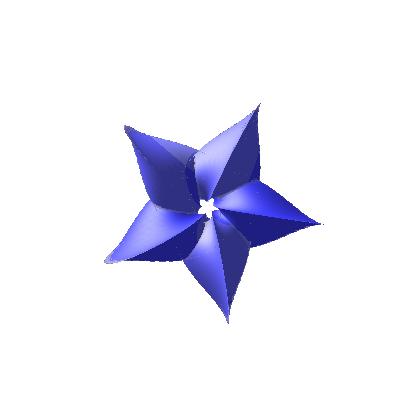}
\end{minipage}
\begin{minipage}[]{0.45\linewidth}
\includegraphics[width=\linewidth]{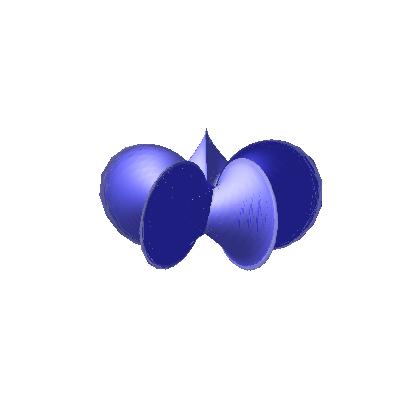}
\end{minipage}
\caption{$a=\sqrt3$ and $b=\sqrt2$.}
\end{figure}

\section{Constant angle surfaces}

In this section we will complete our classification of helicoidal flat 
surfaces in $\Sp^3$, by establishing our second main theorem, that
can be seen as a converse of Theorem \ref{teo:main1}. It is well
known that the Hopf map $h:\Sp^3\to\Sp^2$ is a Riemannian 
submersion and the standard orthogonal basis of $\Sp^3$
\[
E_1(z,w)=i(z,w), \ \ E_2(z,w)=i(-\overline w,\overline z), \ \ 
E_3(z,w)=(-\overline w,\overline z)
\]
has the property that $E_1$ is vertical and $E_2$, $E_3$ are horizontal.
The vector field $E_1$, usually called the Hopf vector field, is an unit
Killing vector field. 

\vspace{.2cm}

Constant angle surface in $\Sp^3$ are those surfaces whose its
unit normal vector field makes a constant angle with the Hopf vector
field $E_1$. The next result states that flatness of a helicoidal surface
in $\Sp^3$ turns out to be equivalent to constant angle surface.

\begin{prop}\label{prop:CAS}
A helicoidal surface in $\Sp^3$, locally parametrized by 
\eqref{eq:param-helicoidal} and with the profile curve $\gamma$ parametrized
by \eqref{eq:param-gamma}, is a flat surface if and only if it is a 
constant angle surface.
\end{prop}
\begin{proof}
Let us consider the Hopf vector field 
\[
E_1(x_1, x_2, x_3, x_4) = (-x_2, x_1, -x_4, x_3),
\]
and let us denote by $\nu$ the angle between $E_1$ and the normal vector
field $N$ along the surface given in \eqref{normal-field}. Along the 
parametrization \eqref{eq:param-helicoidal}, we can write the vector field
$E_1$ as
\[
E_1(X(t,s)) = \Ph(-x_2, x_1,0,x_3).
\]
Then, since $\Ph\in O(4)$, we have
\[
\pl N,E_1\pr(t,s)=\pl N,E_1\pr(s)=
(\beta-\alpha) \frac{x_3 x_3'}{\sqrt{\beta^2x_3^2 + \alpha^2 (x_3')^2}}.
\]
By considering the parametrization \eqref{eq:param-gamma} for the
profile curve $\gamma$, the angle $\nu=\nu(s)$ between $N$ and $E_1$
is given by
\begin{eqnarray}\label{eq:cosnu}
\cos \nu (s) = (\beta-\alpha) \frac{\varphi' \sin\varphi \cos\varphi}
{\sqrt{\beta^2\sin^2 \varphi + \alpha^2 (\varphi')^2 \cos^2 \varphi}}.
\end{eqnarray}
By taking the derivative in \eqref{eq:cosnu}, we have
\[
\frac{d}{ds}(\cos\nu(s))=\frac{(\beta-\alpha)\big(\beta^2\varphi''\sin^3 
\varphi\cos\varphi - \beta^2(\varphi')^2 \sin^4 \varphi + 
\alpha^2 (\varphi')^2 \cos^4 \varphi\big)}
{\big(\beta^2\sin^2\varphi+\alpha^2(\varphi')^2 \cos^2 \varphi \big)^{\frac{3}{2}}},
\]
and the conclusion follows from the Proposition \ref{prop:HSF}.
\end{proof}

Given a number $\epsilon>0$, let us recall that the Berger sphere 
$\Sp^3_\epsilon$ is defined as the sphere $\Sp^3$ endowed with the 
metric
\begin{eqnarray}\label{eq:berger}
\pl X,Y\pr_\epsilon=\pl X,Y\pr+(\epsilon^2-1)\pl X,E_1\pr\pl Y,E_1\pr,
\end{eqnarray}
where $\pl,\pr$ denotes de canonical metric of $\Sp^3$. We define 
constant angle surface in $\Sp^3_\epsilon$ in the same way that
in the case of $\Sp^3$. Constant angle surfaces in the Berger spheres
were characterized by Montaldo and Onnis \cite{MO}. More precisely,
if $M$ is a constant angle surface in the Berger sphere, with constant
angle $\nu$, then there exists a local parametrization $F(u,v)$ given by
\begin{eqnarray}\label{eq:paramMO}
F(u,v)=A(v)b(u),
\end{eqnarray}
where
\begin{eqnarray}\label{eq:curveb}
b(u)=\big(\sqrt{c_1}\cos(\alpha_1u),\sqrt{c_1}\sin(\alpha_1u),
\sqrt{c_2}\cos(\alpha_2u),\sqrt{c_2}\sin(\alpha_2u)\big)
\end{eqnarray}
is a geodesic curve in the torus $\Sp^1(\sqrt{c_1})\times\Sp^1(\sqrt{c_2})$,
with
\[
c_{1,2}=\frac{1}{2}\mp\frac{\epsilon\cos\nu}{2\sqrt{B}},\ 
\alpha_1=\frac{2B}{\epsilon}c_2, \ \alpha_2=\frac{2B}{\epsilon}c_1, \
B=1+(\epsilon^2-1)\cos^2\nu,
\]
and 
\begin{eqnarray}\label{eq:xi_i}
A(v)=A(\xi,\xi_1,\xi_2,\xi_3)(v)
\end{eqnarray}
is a $1$-parameter family of $4\times4$ orthogonal matrices given by
\[
A(v)=A(\xi)\cdot\tilde A(v),
\]
where
\[
A(\xi)=\left(
\begin{array}{cccc}
1 & 0 & 0 & 0 \\
0 & 1 & 0 & 0 \\
0 & 0 & \sin\xi & \cos\xi \\
0 & 0 & -\cos\xi & \sin\xi
\end{array}
\right)
\]
and
\[
\tilde A(v)=\left(
\begin{array}{rrrr}
\cos\xi_1\cos\xi_2 & -\cos\xi_1\sin\xi_2 & \sin\xi_1\cos\xi_2 & -\sin\xi_1\sin\xi_3 \\
\cos\xi_1\sin\xi_2 & -\cos\xi_1\cos\xi_2 & \sin\xi_1\sin\xi_3 & \sin\xi_1\cos\xi_3 \\
-\sin\xi_1\cos\xi_3 & \sin\xi_1\sin\xi_3 & \cos\xi_1\cos\xi_2 & \cos\xi_1\sin\xi_2 \\
\sin\xi_1\sin\xi_3 & -\sin\xi_1\cos\xi_3 & -\cos\xi_1\sin\xi_2 & \cos\xi_1\cos\xi_2
\end{array}
\right),
\]
$\xi$ is a constant and the functions $\xi_i(v)$, $1\leq i\leq 3$, satisfy
\begin{eqnarray}\label{eq:xis}
\cos^2(\xi_1(v))\xi_2'(v)-\sin^2(\xi_1(v))\xi_3'(v)=0.
\end{eqnarray}

In the next result we obtain another relation between the function $\xi_i$,
given in \eqref{eq:xi_i}, and the angle function $\nu$.

\begin{prop}
{\em The functions $\xi_i(v)$, given in \eqref{eq:xi_i}, satisfy the following relation:
\begin{eqnarray}\label{eq:relxi_inu}
(\xi_1'(v))^2+(\xi_2'(v))^2\cos^2(\xi_1(v))+(\xi_3'(v))\sin^2(\xi_1(v))=
\sin^2\nu,
\end{eqnarray}
where $\nu$ is the angle function of the surface $M$.}
\end{prop}
\begin{proof}
With respect to the parametrization $F(u,v)$, given in \eqref{eq:paramMO},
we have
\[
F_v=A'(v)\cdot b(u)=A(\xi)\cdot\tilde A'(v)\cdot b(u).
\]
We have $\pl F_v,F_v\pr=\sin^2\nu$ (cf. \cite{MO}). On the other hand, 
if we denote by $c_1$, $c_2$, $c_3$, $c_4$ the columns of $\tilde A$, 
we have
\[
\pl F_v,F_v\pr\vert_{u=0}=g_{11}\pl c_1',c_1'\pr+g_{33}\pl c_3',c_3'\pr.
\]
As $\pl c_1',c_1'\pr=\pl c_3',c_3'\pr$, $\pl c_1',c_3'\pr=0$ and 
$g_{11}+g_{33}=1$, a straightforward computation gives
\begin{eqnarray*}
\sin^2\nu &=& \pl F_v,F_v\pr = (g_{11}+g_{33})\pl c_1',c_1'\pr \\
&=& (\xi_1'(v))^2+(\xi_2'(v))^2\cos^2(\xi_1(v))+(\xi_3'(v))\sin^2(\xi_1(v)),
\end{eqnarray*}
and we conclude the proof.
\end{proof}

\begin{thm}
Let $M$ be a helicoidal flat surface in $\Sp^3$, locally parametrized by 
\eqref{eq:param-helicoidal}, and whose profile curve $\gamma$ is given
by \eqref{eq:param-gamma}. Then $M$ admits a new local 
parametrization such that the fundamental forms are given as in
\eqref{eq:forms} and $\omega$ is a linear function.
\end{thm}
\begin{proof}
Consider the unit normal vector field $N$ associated to the local
parametrization $X$ of $M$ given in \eqref{eq:param-helicoidal}. 
From Proposition \ref{prop:CAS}, the angle between $N$ and the
Hopf vector field $E_1$ is constant. Hence, it follows from 
\cite{MO} (Theorem 3.1) that $M$ can be locally parametrized as
in \eqref{eq:paramMO}.
By taking $\epsilon=1$ in \eqref{eq:berger}, we can reparametrize the curve 
$b$ given in \eqref{eq:curveb} in such a way that the new curve is a base
curve $\gamma_a$.
In fact, by taking $\epsilon=1$, we obtain $B=1$, and so 
$\alpha_1=2c_2$ and $\alpha_2=2c_1$. This implies that
$\|b'(u)\|=2\sqrt{c_1c_2}$, because $c_1+c_2=1$. Thus, by writing
$s=2\sqrt{c_1c_2}$, the new parametrization of $b$ is given by
\[
b(s)=\frac{1}{\sqrt{1+a^2}} \left(a\cos\frac{s}{a},a\sin\frac{s}{a}, 
\cos(as),\sin(as)\right), 
\]
where $a=\sqrt{c_1/c_2}$. On the other hand, we have
\[
A(v)\cdot b(s)=A(\xi)X(v,s),
\]
where $X(v,s)$ can be written as
\[
X(v,s)=\frac{1}{\sqrt{1+a^2}}(x_1,x_2,x_3,x_4),
\]
with
\begin{eqnarray}\label{eq:coef_xi}
\begin{array}{rcl}
x_1 &=& a\cos\xi_1\cos\left(\dfrac{s}{a}+\xi_2\right)+\sin\xi_1\cos(as+\xi_3), \\
x_2 &=& a\cos\xi_1\sin\left(\dfrac{s}{a}+\xi_2\right)+\sin\xi_1\sin(as+\xi_3), \\
x_3 &=& -a\sin\xi_1\cos\left(\dfrac{s}{a}-\xi_3\right)+\cos\xi_1\cos(as-\xi_2), \\
x_4 &=& -a\sin\xi_1\sin\left(\dfrac{s}{a}-\xi_3\right)+\cos\xi_1\sin(as-\xi_2). \\
\end{array} 
\end{eqnarray}
On the other hand, the product $\Ph\cdot X(v,s)$ can be written as
\[
\Ph\cdot X(v,s)=\frac{1}{\sqrt{1+a^2}}(z_1,z_2,z_3,z_4),
\]
where
\begin{eqnarray}\label{eq:coef_zi}
\begin{array}{rcl}
z_1 &=& a\cos\xi_1\cos\left(\dfrac{s}{a}+\xi_2+\alpha t\right)
+\sin\xi_1\cos\left(as+\xi_3+\alpha t\right), \\
z_2 &=& a\cos\xi_1\sin\left(\dfrac{s}{a}+\xi_2+\alpha t\right)
+\sin\xi_1\sin\left(as+\xi_3+\alpha t\right), \\
z_3 &=& -a\sin\xi_1\cos\left(\dfrac{s}{a}-\xi_3+\beta t\right)
+\cos\xi_1\cos\left(as-\xi_2+\beta t\right), \\
z_4 &=& -a\sin\xi_1\sin\left(\dfrac{s}{a}-\xi_3+\beta t\right)
+\cos\xi_1\sin\left(as-\xi_2+\beta t\right).
\end{array} 
\end{eqnarray}
As the surface is helicoidal, we have
\[
\Ph\cdot X(v,s)=X(v(t),s(t)),
\]
for some smooth functions $v(t)$ and $s(t)$, which satisfy the
following equations:
\begin{eqnarray}
\xi_2(v(t)) + \dfrac{s(t)}{a} = \xi_2(v)+ \dfrac{s}{a}+\alpha t, \label{eq:xi2-alpha} \\
\xi_3(v(t)) + a s(t) = \xi_3 (v) + a s + \alpha t, \label{eq:xi3-alpha} \\
\dfrac{s(t)}{a} - \xi_3(v(t)) = \dfrac{s}{a} - \xi_3(v) + \beta t, \label{eq:xi3-beta} \\
a s(t) - \xi_2(v(t)) = as - \xi_2(v) + \beta t. \label{eq:xi2-beta}
\end{eqnarray}
It follows directly from \eqref{eq:xi2-alpha} and \eqref{eq:xi2-beta} that 
\begin{equation}
s(t) = s + \dfrac{a(\alpha+\beta)}{a^2+1} t. \label{eq:s(t)}
\end{equation}
Note that the same conclusion is obtained by using \eqref{eq:xi3-alpha}
and \eqref{eq:xi3-beta}. By substituting the expression of $s(t)$ given in 
\eqref{eq:s(t)} on the equations \eqref{eq:xi2-alpha} -- \eqref{eq:xi2-beta},
one has
\begin{eqnarray}
\xi_2(v(t))= \xi_2(v) + \left( \dfrac{a^2 \alpha - \beta}{a^2+1} \right) t, \label{eq:xi2(v(t))}  \\
\xi_3(v(t))= \xi_3(v) + \left( \dfrac{\alpha - a^2 \beta}{a^2+1} \right) t. 
\label{eq:xi3(v(t))} 
\end{eqnarray}
From now on we assume that $v'(t)\neq0$ since, otherwise, we would
have
\[
\frac{s(t)}{a}=\frac{s}{a}+\alpha t=\frac{s}{a}+\beta t
\quad\text{and}\quad
as(t)=as+\alpha t=as+\beta t.
\]
But the equalities above imply that $a^2=1$, which contradicts the
definition of base curve in \eqref{eq:base}. Thus, it follows from 
\eqref{eq:xi2(v(t))} and \eqref{eq:xi3(v(t))} that
\begin{eqnarray}\label{eq:xi_2exi_3}
\xi_2'=\frac{a^2 \alpha-\beta}{a^2+1}\cdot\frac{1}{v'} \quad\text{and}\quad
\xi_3'=\frac{\alpha-a^2 \beta}{a^2+1}\cdot\frac{1}{v'}.
\end{eqnarray}
Therefore, from \eqref{eq:xis} and \eqref{eq:xi_2exi_3} we obtain
\begin{eqnarray}\label{eq:relxi_1}
\cos^2(\xi_1(v))(a^2\alpha-\beta)=\sin^2(\xi_1(v))(\alpha-a^2\beta).
\end{eqnarray}
As $a>1$, one has $a^2\alpha-\beta\neq0$ or $\alpha-a^2\beta\neq0$, and
we conclude from \eqref{eq:relxi_1} that $\xi_1(v)$ is constant. In this 
case, there is a constant $b>1$ such that $\cos^2\xi_1=\dfrac{b^2}{1+b^2}$
and $\sin^2 \xi_1 = \dfrac{1}{1+b^2}$. Therefore, it follows from \eqref{eq:xis}
that 
\begin{eqnarray}\label{eq:xi2-xi3}
\xi_2(v)= \dfrac{1}{b^2} \xi_3(v)+d,
\end{eqnarray}
for some constant $d$. On the other hand, if $\cos\xi_1\neq0$, it follows
from \eqref{eq:xis} that
\begin{eqnarray}\label{eq:xi_2xi_3}
(\xi_2'(v))^2=\tan^4\xi_1\cdot(\xi_3'(v))^2.
\end{eqnarray}
By substituting \eqref{eq:xi_2xi_3} in \eqref{eq:relxi_inu} we obtain
\[
\tan^2\xi_1\cdot(\xi_3'(v))^2=\sin^2\nu,
\]
ant this implies that we can choose $\xi_3(v)=bv$, and from \eqref{eq:xi3(v(t))} 
we obtain 
\begin{eqnarray}\label{eq:v(t)}
v(t)=v+\frac{\alpha-a^2\beta}{b(a^2+1)}t.
\end{eqnarray}
Moreover, from \eqref{eq:xi2-xi3}, the equation 
$\xi_2(v(t)) = \dfrac{1}{b^2} \xi_3(v(t))+d$ implies that 
\begin{eqnarray}\label{eq:b-alpha-beta}
\dfrac{1}{b^2} = \dfrac{a^2 \alpha - \beta}{\alpha-a^2 \beta},
\end{eqnarray}
and from \eqref{eq:b-alpha-beta} we obtain the same 
relation \eqref{eq:alpha-beta} between $\alpha$ and $\beta$. This
relation, when substituted  in \eqref{eq:s(t)} and \eqref{eq:v(t)},
gives
\[
s(t)=s+\frac{a(b^2-1)}{a^2b^2-1}\beta t \quad\text{and}\quad
v(t)=v+\frac{b(1-a^2)}{a^2b^2-1}\beta t,
\]
that coincide with the expressions in \eqref{eq:z(t)w(t)}. Finally, from
the relation \eqref{eq:xi2-xi3} we obtain $\xi_2(v)=\frac{v}{b}-\frac{\pi}{2}$.
By taking $\xi=\frac{\pi}{2}$ and $\xi_1(v)=\arcsin\left(\frac{1}{\sqrt{1+b^2}}\right)$,
the new parametrization $F(u,v)$ thus obtained coincides with $Y(u,v)$
given in \eqref{eq:expY}, up to isometries of $\Sp^3$ and linear
reparametrization. The conclusion follows from Theorem \ref{teo:main1}.
\end{proof}

\section{Conformally flat hypersurfaces}

In this section, it will presented an application of the classification result 
for helicoidal flat surfaces in $\mathbb{S}^3$ in a geometric description 
for conformally flat hypersurfaces in four-dimensional space forms.

The problem of classifying conformally flat hypersurfaces in space forms
has been investigated for a long time, with special attention on $4$-dimensional
space forms. In fact, any surface in $\R^3$ is conformally flat, since it can be parametrized by isothermal coordinates. On the other hand, Cartan \cite{Cartan}
gave a complete classification of conformally flat hypersurfaces into a 
$(n+1)$-dimensional space form, with $n+1\geq5$. Such hypersurfaces
are quasi-umbilic, i.e., one of the principal curvatures has multiplicity at least
$n-1$. In the same paper, Cartan showed that the quasi-umbilic surfaces are conformally flat, but the converse does not hold. Since then, there has been
an effort to obtain a classification of hypersurfaces with three distinct principal curvatures.

Lafontaine \cite{Lafontaine} considered hypersurfaces of type 
$M^3 = M^2 \times I\subset\R^4$ and obtained the following classes of 
conformally flat hypersurfaces: (a) $M^3$ is a cylinder over a surface, where
$M^2 \subset \R^3$ has constant curvature; (b) $M^3$ is a cone over a 
surface in the sphere, where $M^2 \subset\Sp^3$ has with constant curvature; 
(c) $M^3$ is obtained by rotating a constant curvature surface of the hyperbolic space and $M^2 \subset \mathbb{H}^3 \subset \R^4$, where $\mathbb{H}^3$
is the half space model (see \cite{Suyama2} for more details).

Hertrich-Jeromin \cite{Jeromin1} established a correspondence between 
conformally flat hypersufaces in space forms, with three distinct principal
curvatures, and solutions $(l_1, l_2, l_3) : U \subset \R^3 \rightarrow \R$ 
for the Lam\'e's system \cite{reflame}
\begin{equation}
\begin{array}{rcl}
	l_{i,x_jx_k} - \dfrac{l_{i,x_j} l_{j,x_k}}{l_j} - \dfrac{l_{i,x_k} l_{k,x_j}}{l_k} &=& 0,  \\
	\left( \dfrac{l_{i,x_j}}{l_j} \right)_{,x_j} + \left( \dfrac{l_{j,x_i}}{l_i} \right)_{,x_i} + \dfrac{l_{i,x_k} l_{j,x_k}}{l_k^2} &=& 0,
\end{array} \label{lame}
 \end{equation} 
where $i, j, k$ are distinct indices that satisfies the condition
\begin{equation} \label{guich}
l_1^2 - l_2^2 + l_3^2 = 0 
\end{equation}
known as Guichard condition. In this case, the correspondent coformally
flat hypersurface in $M^4_K$ is parametrized by curvature lines, with 
induced metric given by
\[
g = e^{2u} \left\{ l_1^2 (dx_1)^2 + l_1^2 (dx_1)^2 + l_3^2 (dx_3)^2  \right\}.
\]
 
In \cite{santos} the second author and Tenenblat obtained solutions of Lam\'e's system (\ref{lame}) that are invariant under symmetry groups. Among the solutions, there are those that are invariant under the action of the 2-dimensional subgroup of translations and dilations and depends only on two variables:
\begin{enumerate}
\item[(a)] $l_1 = \lambda_1, \; l_2 = \lambda_1 \cosh (b\xi + \xi_0), \; l_3 = \lambda_1 \sinh (b\xi + \xi_0)$, where $\xi = \alpha_2 x_2 + \alpha_3 x_3$, $\alpha_2^2 + \alpha_3^2 \neq 0$ and $b,\, \xi_0\in \R$ ; 
\item[(b)] $ l_2 = \lambda_2, \; l_1 = \lambda_2 \cos \varphi(\xi) , \; l_3 = \lambda_2 \sin \varphi(\xi) $, where $\xi = \alpha_1 x_1 + \alpha_3 x_3$, $\alpha_1^2 + \alpha_3^2 \neq 0$ and $\varphi$ is one  of the following functions:
   \begin{enumerate}
      \item[(b.1)] $\varphi(\xi) = b \xi + \xi_0$, if $\alpha_1^2 \neq \alpha_3^2$, where $\xi_0, \,b\in \R$;
      \item[(b.2)] $\varphi$ is any function of $\xi$, if $\alpha_1^2 = \alpha_3^2$;
   \end{enumerate}
\item[(c)]  $l_3 = \lambda_3, \; l_2 = \lambda_3 \cosh (b\xi + \xi_0), \; l_1 = \lambda_3 \sinh (b\xi + \xi_0)$, where $\xi = \alpha_1 x_1 + \alpha_2 x_2$, $\alpha_1^2 + \alpha_2^2 \neq 0$ and   $b,\,\xi_0\in\R $. 
\end{enumerate}
 It is known (see \cite{Suyama2}) that the solutions that do not depend on 
one of the variables are associated to the products given by Lafontaine.
For the solutions given in (b), further geometric solutions can be obtained
with the classification result for helicoidal falt surfaces in $\mathbb{S}^3$.
These solutions are associated to  conformally flat hypersurfaces that are
conformal to the products $M^2 \times I \subset \R^4$ given by 
\[
M^2 \times I = \left\{ tp:0<t<\infty, p\in M^2 \subset \mathbb{S}^3 \right\},
\]
where $M^2$ is a flat surface in $S^3$, parametrized by lines of curvature,
whose first and second fundamental forms are given by 
\begin{equation}
\begin{array}{rcl} \label{firstffsphere}
I &=& \sin^2 (\xi + \xi_0) dx_1^2 + \cos^2 (\xi + \xi_0)  dx_3^2,  \\
II &=& \sin (\xi + \xi_0) \cos (\xi + \xi_0) (dx_1^2 - dx_3^2),
\end{array}
\end{equation}
which are, up to a linear change of variables, the fundamental forms that
are considered in this paper. Therefore, as an application of the 
characterization of helicoidal flat surfaces in terms of first and second
fundamental forms, one has the following theorem:

\begin{thm} \label{characterization.helicoidal}
Let $ l_2 = \lambda_2, \; l_1 = \lambda_2 \cos \xi + \xi_0 , \; 
l_3 = \lambda_2 \sin \xi + \xi_0 $ be solutions of the Lam\'e's sytem, where
$\xi = \alpha_1 x_1 + \alpha_3 x_3$ and $\alpha_1, \, \alpha_3, \, \lambda_2, \, \xi_0$ are real constants with $\alpha_1\cdot\alpha_3 \neq 0$. Then the associated conformally flat hypersurfaces are conformal to the product, $M^2 \times I$,
where $M^2 \subset \mathbb{S}^3$ is locally congruent to helicoidal flat surface.
\end{thm}

\bibliographystyle{acm}
\bibliography{refs_helicoidal-sphere}

\end{document}